\documentclass[12pt]{amsart}
\usepackage[letterpaper,left=2.5cm,top=2.5cm,bottom=2.75cm, 
right=2.5cm]{geometry}

\usepackage{amssymb,amsmath,mathrsfs,amscd, esint,amsthm}
\usepackage{graphicx, xcolor,amsfonts,times,enumerate}

\newcommand{\R}{\mathbb{R}}
\newcommand{\Rn}{{\R^n}}

\newcommand{\E}{\mathbb{E}}

\newcommand{\cS}{{\mathscr{S}}}
\newcommand{\cA}{{\mathscr{A}}}

\newcommand{\cQ}{{\mathcal{Q}}}
\newcommand{\cB}{{\mathcal{B}}}
\newcommand{\cP}{{\mathcal{P}}}
\newcommand{\cR}{{\mathcal{R}}}
\newcommand{\cI}{{\mathcal{I}}}
\newcommand{\cO}{{\mathcal{O}}}

\newcommand{\cF}{{\mathcal{F}}}
\newcommand{\cW}{{\mathcal{W}}}

\newcommand{\Stilde}{{\widetilde{S}}}

\newcommand{\Rtilde}{{\widetilde{R}}}

\newcommand{\fast}{f^\ast}

\newcommand{\gast}{g^\ast}

\newcommand{\Linfty}{{L^\infty}}

\newcommand{\Loneloc}{{L^1_{\text{loc}}}}

\newcommand{\ra}{\rightarrow}

\DeclareMathOperator*{\essinf}{ess\,inf}

\newcommand{\bmo}{{\textnormal{BMO}}} 

\def\BMO#1#2{\textnormal{BMO}_{#1}^{#2}}

\def\BLO#1#2{\textnormal{BLO}_{#1}^{#2}}

\def\XXint#1#2#3{{\setbox0=\hbox{$#1{#2#3}{\int}$ }
\vcenter{\hbox{$#2#3$ }}\kern-.57\wd0}}

\newtheorem{theorem}{Theorem}[section]
\newtheorem{lemma}[theorem]{Lemma}

\newtheorem*{corollary*}{Corollary}

\theoremstyle{definition}
\newtheorem{definition}[theorem]{Definition}

\theoremstyle{remark}

\numberwithin{equation}{section}

\begin{document}

\newpage

\title[Mean oscillation bounds on rearrangements]
{Mean oscillation bounds on rearrangements}

\author[Burchard]{Almut Burchard}
\address{(A.B.) University of Toronto, Department of Mathematics, Toronto, ON M5S 2E4, Canada}
\curraddr{}
\email{almut@math.toronto.edu}

\author[Dafni]{Galia Dafni}
\address{(G.D.) Concordia University, Department of Mathematics and Statistics, Montr\'{e}al, QC H3G 1M8, Canada}
\curraddr{}
\email{galia.dafni@concordia.ca}
\thanks{A.B. was partially supported by Natural Sciences and Engineering Research Council (NSERC) of Canada. G.D. was partially supported by the Natural Sciences and Engineering Research Council (NSERC) of Canada and the Centre de recherches math\'{e}matiques (CRM). R.G. was partially supported by the Centre de recherches math\'{e}matiques (CRM), the Institut des sciences math\'{e}matiques (ISM), and the Fonds de recherche du Qu\'{e}bec -- Nature et technologies (FRQNT)}

\author[Gibara]{Ryan Gibara}
\address{(R.G.) Universit\'{e} Laval, D\'{e}partement de math\'{e}matiques et de statistique, Qu\'{e}bec, QC G1V 0A6, Canada}
\curraddr{}
\email{ryan.gibara@gmail.com}


\begin{abstract}

We use geometric arguments to prove explicit bounds on the mean oscillation for two important rearrangements on $\Rn$.  For the decreasing rearrangement $f^*$ of a rearrangeable function $f$ of bounded mean oscillation (BMO) on cubes, we improve a classical inequality of Bennett--DeVore--Sharpley, $\|f^*\|_{\bmo(\R_+)}\leq C_n \|f\|_{\bmo(\R^n)}$, by showing the growth of $C_n$ in the dimension $n$ is not exponential but  at most of the order of $\sqrt{n}$.  This is achieved by comparing cubes to a family of rectangles for which one can prove a dimension-free Calder\'{o}n--Zygmund decomposition. By comparing cubes to a family of polar rectangles, we provide a first proof that an analogous inequality holds for the symmetric decreasing rearrangement, $Sf$. 
\end{abstract}

\maketitle


\section{{\bf  Introduction}}
\label{sec:intro}


Equimeasurable rearrangements are used in analysis and mathematical physics to reduce extremal problems involving functions on higher-dimensional spaces to problems for functions of a single variable. Here, we consider the action of two classical rearrangements on the space BMO, consisting of functions with bounded mean oscillation. The rearrangements are {\em symmetrization} (symmetric decreasing 
rearrangement), which replaces a given 
function $f$ on $\Rn$ with 
a radially decreasing function $Sf$, 
and the {\em decreasing rearrangement}, which replaces 
$f$ with a decreasing function $f^*$ on $\R_+$.

Symmetrization, when applicable, offers a direct path to geometric inequalities in functional analysis. To name a few examples, in the first computation of the sharp constants in the Sobolev inequality for $\|\nabla f\|_p$, due to Talenti~\cite{Talenti} and Aubin~\cite{Aubin}, symmetrization is used to reduce the problem to radial functions. A similar construction appears in the work of Lieb~\cite{BL} on extremals of the 
Hardy--Littlewood--Sobolev inequality. Sharp rearrangement inequalities, which also characterize equality cases, are crucial for identifying extremals. 
They can also be used constructively, according to the 
Competing Symmetries principle of Carlen--Loss~\cite{CL}, 
to find sequences with rearrangements that converge strongly 
to a particular extremal.

Many applications rely on classical symmetrization 
inequalities in Lebesgue and Sobolev spaces. 
Rearrangements generally preserve $L^p$-norms and 
contract $L^p$-distances for $1\le p\le\infty$. 
Symmetrization also improves the modulus of continuity. 
By the P\'olya--Szeg\H{o} 
inequality, it decreases $\|\nabla f\|_p$ for $p\ge 1$, and 
by the Riesz--Sobolev inequality it increases norms in 
certain negative Sobolev spaces $\|f\|_{H^{-s}}$ for $0<s<\frac n 2$.
These inequalities indicate that  symmetrization reduces the 
overall oscillation of $f$. They contain the classical 
isoperimetric inequality as the limiting case where $f$ is the 
characteristic function of a set of finite volume.

The decreasing rearrangement, on the other hand, is 
determined by the distribution function, which contains 
no information about the shape of level sets. It can be defined 
on a general measure space and is widely applied, for example, 
on metric spaces endowed with a doubling measure. The 
decreasing rearrangement plays an important role in interpolation 
theory as many familiar function spaces are invariant under equimeasurable rearrangements, including the Lebesgue, Lorentz, and Orlicz spaces (see~\cite{bs}). 

We study the rearrangements of functions of bounded mean oscillation (BMO), a condition introduced by John and Nirenberg in~\cite{jn}. By definition, a locally integrable function $f$ is in $\bmo(\R^n)$ if it satisfies a uniform bound on its mean oscillation over cubes. Unlike $L^p$-spaces, BMO is not invariant under equimeasurable rearrangements (for instance, $f$ defined by $f(x)=(-\log |x|)_+$ lies in $\bmo(\R)$, but $g$ defined by $g(x)=f(\frac12 x)$ for $x>0$ and $g(x)=0$ for $x\le 0$ does not.) Moreover, recent results~\cite{dhky} on the John--Nirenberg space $JN_p$, a variant of BMO introduced in~\cite{jn}, show that spaces defined by mean oscillation need not be preserved by the decreasing rearrangement. 

The work of Bennett--DeVore--Sharpley in~\cite{bdvs} implies, as shown in~\cite{ko2}, that the decreasing rearrangement of a function $f$ in $\bmo(\Rn)$ belongs to $\bmo(\R_+)$. Moreover, there are constants $C_n$ (depending only on dimension) such that
\begin{equation}
\label{eq-bds}
\|\fast\|_{\bmo}\leq C_n \|f\|_{\bmo}\,.
\end{equation}
The explicit bound on the constant arising from \cite{bdvs} is $C_n\le 2^{n+5}$. When $n=1$, the results of Klemes~\cite{kl}, along with subsequent steps taken by Korenovskii~\cite{ko1}, imply that $C_1=1$. (We neglect, for the moment, the distinction between $\Rn$ and finite cubes, see Section~\ref{sec:bounds}.)

Our first main result is that the exponential dependence on the dimension in the bound of Bennett--DeVore--Sharpley can be eliminated.
\begin{theorem}
\label{thm-newbound}
For rearrangeable $f\in\bmo(\R^n)$, Eq.~\eqref{eq-bds} holds with $C_n\le 2(1+2\sqrt{n-1})$.
\end{theorem}

To achieve this, we exhibit a basis 
$\cW$ of rectangles in $\Rn$, comparable with cubes, such that
\begin{equation}
\label{eq-newbound}
\|\fast\|_{\bmo}\leq 2\|f\|_{\BMO{\cW}{}}\, ,
\end{equation}
where $\|f\|_{\BMO{\cW}{}}$ denotes
the supremum of the mean oscillation of $f$ over these 
rectangles.  Moreover,
\begin{equation} 
\label{eq:falsecompare}
\|f\|_{\bmo}\le \|f\|_{\BMO{\cW}{}}\le (1+2\sqrt{n-1})\|f\|_{\bmo}\,.
\end{equation}

The family of rectangles $\cW$ consists of the {\em false cubes} introduced by Wik in his work on the constants in the John-Nirenberg inequality~\cite{wik}. The collection of false cubes has 
the property that each false cube can be bisected into two smaller false cubes, and in Section~\ref{sec:bounds} we prove Eq.~\eqref{eq-newbound} for any such {\em bisection basis}. The equivalence of the seminorms 
$\|\cdot\|_{\BMO{\cW}{}}$ and $\|\cdot\|_{\BMO{}{}}$
up to a constant of order $O(\sqrt{n})$
was first proved by Wik (Theorem 2 in~\cite{wik}), using a combinatorial 
argument and a slightly different definition of the norms.
Here, we give a simple probabilistic proof for Eq.~\eqref{eq:falsecompare}.

While the sharp constant $C_1=1$ of Klemes-Korenovskii holds in dimension $n=1$, the question of the best constants $C_n$ in Eq.~\eqref{eq-bds}, and whether such constants can be dimension-free, remains open for $n>1$.  However, sharp results are known for other BMO-spaces.
For the smaller space $\BMO{\cR}{}$, called 
anisotropic or strong BMO, consisting of functions of bounded 
mean oscillation on rectangles, a sharp constant of $1$ was obtained by Korenovskii in ~\cite{ko3} by using a higher-dimensional version of the Riesz Rising Sun lemma for rectangles - see also Korenovskii--Lerner--Stokolos~\cite{kls}. 

Korenovskii shows that for a rectangle $R_0$,
\begin{equation}
\label{eq-KK}
\|\fast\|_{\bmo(0,|R_0|)}\leq \|f\|_{\BMO{\cR}{}(R_0)}\,.
\end{equation} 
The inequalities of Klemes and Korenovskii extend to the entire space $\Rn$ (using the techniques in Section~\ref{sec-rect_to_cubes} below). On the other hand, for the larger space dyadic BMO,  working with the $L^2$-mean oscillation, Stolyarov, Vasyunin, and Zatitskiy~\cite{svz} obtained a sharp dimension-dependence of $2^{n/2}$ in the boundedness of the decreasing rearrangement.

There has been recent interest in dimension-free bounds for averaging operators~\cite{bmsw1, bmsw2}, and  dimension-free constants in the John--Nirenberg inequality~\cite{css}. Improved bounds on the decreasing  rearrangement on BMO yield improved constants in the John--Nirenberg inequality, by reducing to the one-dimensional case: from Eq.~\eqref{eq-KK}, Korenovskii~\cite{ko3} obtained sharp constants for the John--Nirenberg inequality in anisotropic BMO on a rectangle. 

It is apparent from the results discussed above that the geometry of the sets over which mean oscillation is measured plays a crucial role in the properties of the corresponding BMO-space. In previous work, two of the authors~\cite{dg} introduce the space $\BMO{\cS}{}$, consisting of functions of bounded mean oscillation on sets $S$, called {\em shapes}, forming a {\em basis} $\cS$ --- see Section~\ref{sec:prelim} for the relevant definitions. The proof of Theorem~\ref{thm-newbound} demonstrates that equivalent bases of shapes provide powerful tools for geometric analysis in BMO.

\smallskip 
Comparison of shapes also plays a key role in the relation between $\|Sf\|_{\bmo(\R^n)}$ and $\|f^*\|_{\bmo(\R_+)}$, shown in Section~\ref{sec:SDR}. Geometrically, the symmetric decreasing rearrangement $Sf$ is linked with $\fast$ by means of the formula $Sf(x)=\fast(\omega_n|x|^n)$. We give the following bi-Lipschitz equivalence between $Sf$ and $f^*$.

\begin{theorem}
\label{thm-SDRequivalent}
If $f_1,f_2$ are rearrangeable functions in $\bmo(\R^n)$, then 
\begin{equation}
\label{eq-SDRequivalent}
2^{-2n}\omega_n\|f_1^*-f_2^*\|_{\bmo}
  \leq \|Sf_1-Sf_2\|_{\bmo}
  \leq n^\frac{n}{2}\omega_n\|f_1^*-f_2^*\|_{\bmo}.
\end{equation}
\end{theorem}

In the proof of Eq.~\eqref{eq-SDRequivalent}, we compare cubes in $\Rn$ with an equivalent basis of shapes $\cA$, comprised of certain 
annular sectors.
The oscillation of $Sf$ over these sectors coincides 
with the mean oscillation of $\fast$ over corresponding intervals. 
The dimension-dependent constants arise from volume factors 
associated when inscribing and circumscribing sectors in 
$\cA$ with balls and cubes.

This theorem makes it possible to transfer results on the decreasing rearrangement $f^*$ to the symmetrization $Sf$. In particular, taking $f_1=f$ and $f_2=0$ in Theorem~\ref{thm-SDRequivalent}, together with Theorem~\ref{thm-newbound}, we see that the symmetric decreasing rearrangement is bounded on $\bmo(\Rn)$:

\begin{corollary*}
There are constants $D_n$ (depending only on dimension) such that if $f\in\bmo(\Rn)$ is rearrangeable, then 
$Sf\in\bmo(\Rn)$ with 
\begin{equation}
\label{eq-SDR}
\|Sf\|_{\bmo}\leq D_n \|f\|_{\bmo}\,.
\end{equation}
\end{corollary*}

One can take $D_n=2(1+2\sqrt{n-1})n^{\frac{n}{2}} \omega_n$. To our knowledge, there are no prior results on symmetrization in $\bmo(\Rn)$ for dimensions $n>1$. In one dimension, Klemes~\cite{kl} proved that the symmetric decreasing rearrangement on the circle satisfies $\|Sf\|_{\bmo} \leq 2\|\fast\|_{\bmo}$, and gives an example for which $\|Sf\|_{\bmo} > \|\fast\|_{\bmo}$. The sharp constants $D_n$ are not known in any dimension.

\section{{\bf Preliminaries}}
\label{sec:prelim}


\subsection{Rearrangements}
Let $f$ be a measurable function on a domain $\Omega\subset\R^n$. 
The {\em distribution function} of $f$ is defined, for $\alpha\geq{0}$, by
$$
\mu_f(\alpha)=|\{x\in \Omega:|f(x)|>\alpha \}|\,.
$$
Here $|\cdot|$ denotes Lebesgue measure. Note that 
$\mu_f:[0,\infty)\rightarrow[0,\infty]$ is 
decreasing (meaning, here and in the rest of the 
paper, nonincreasing) and right-continuous. 

\begin{definition} 
\label{def-fast}
We say that a measurable function $f$ is {\em rearrangeable} if $\mu_f(\alpha)\rightarrow{0}$ as $\alpha\rightarrow\infty$. The {\em decreasing rearrangement} of such a function is the function $\fast:\R_+\rightarrow\R_+$ given by 
$$
\fast(s)= \inf\{\alpha\geq{0}:\mu_f(\alpha)\leq{s} \}\,.
$$
\end{definition}

In other words, the decreasing rearrangement of a function $f$ is the generalized inverse of its distribution function. In the case where $|\Omega|<\infty$, we consider $\fast$ as a function on $(0,|\Omega|)$, since $\fast(s)=0$ for all $s\geq|\Omega|$.

The condition that $f$ is rearrangeable guarantees that the set $\{\alpha\geq{0}:\mu_f(\alpha)\leq{s} \}$ is nonempty for $s>0$ and so $\fast$ is finite on its domain. 
The set $\{\alpha\geq{0}:\mu_f(\alpha)=0 \}$, however, can be empty. If $f$ is bounded, then $\fast$ tends to $\|f\|_{L^\infty}$ as $s\rightarrow{0}^+$; otherwise, $\fast$ is unbounded at the origin. As is the case with the distribution function, $\fast$ is decreasing and right-continuous. Furthermore, $f$ and $f^*$ are equimeasurable in the sense that the distribution function of $f^*$ equals that of $f$ for all $\alpha\geq{0}$.

We will need to use the following elementary form of the {\em Hardy-Littlewood inequality}: for any measurable set $A\subset \Omega$, 
\begin{equation}
\label{R-HL} 
\int_{A}\!|f|\leq\int_0^{|A|}\!\fast\,.
\end{equation}

\begin{definition} 
\label{def-Sf}
Let $f$ be a rearrangeable function on $\R^n$. Its {\em symmetric decreasing rearrangement} $Sf$ is defined by
$$
Sf(x)=\fast(\omega_n|x|^n)\,,\qquad x\in\R^n\setminus \{0\}\,,
$$
where $\omega_n$ is the volume of the unit ball in $\Rn$.
\end{definition}

The symmetric decreasing rearrangement defines a map 
from functions on $\Rn$ to functions on $\Rn$, and from $f^*$ it 
inherits equimeasurability with $f$.  The reader is invited to see \cite{sw} 
for more details on the decreasing rearrangement, and~\cite{Baernstein}
for the symmetric decreasing rearrangement.


\subsection{Bounded mean oscillation}

A {\em shape} is an open set $S\subset\Rn$ with
$0<|S|<\infty$. A {\em basis} of shapes in a domain 
$\Omega\subset\R^n$, then, is a collection $\cS$ of shapes $S\subset{\Omega}$ forming a cover of $\Omega$. Common examples of bases are the collections of all open Euclidean, $\cB$; all finite open cubes with sides parallel to the axes, $\cQ$; and, all finite open rectangles with sides parallel to the axes, $\cR$. In one dimension, these three choices coincide with the collection of all finite open intervals, $\cI$. It is understood that when working on a domain $\Omega$, the notation for the bases above refers to those shapes contained in $\Omega$.

Functions on $\Omega$ are
generally assumed to be real-valued and measurable,
as well as integrable on every shape $S\subset\Omega$ coming from a basis 
$\cS$; such functions are automatically
locally integrable. The {\em mean oscillation} of a function $f$ 
on a shape $S\in\cS$ is defined by
$$
\cO(f,S):=\fint_{S}\!|f-f_S|\,,
$$
where $f_S = \fint_S f$ is the average of $f$ over $S$. It is 
immediate from the definition that $\cO(f+\alpha,S)=\cO(f,S)$ 
for any constant $\alpha$. 

The following inequality allows for the comparison of mean oscillation over different shapes at the cost of the volume ratio: for any  pair of shapes $S\subset\Stilde$,
\begin{equation}
\label{O-subset}
\cO(f,S)\leq \frac{|\Stilde|}{|S|}\, \cO(f,\Stilde)\,.
\end{equation}
This follows from the fact that $\cO(f,S)=2\fint_{S}(f-f_S)_+$,
where $y_+=\max(y,0)$ (see ~\cite{dg}).

\begin{definition}\label{equiv}
Let $\cS$ and $\widetilde{\cS}$ be two bases of shapes in $\Omega$. We say that $\cS$ is {\em equivalent} to $\widetilde{\cS}$, written $\cS\approx\widetilde{\cS}$, if there are constants $c,\tilde{c}>0$ such 
that for every $S\in\cS$ there exists $\Stilde\in\widetilde{\cS}$ with $S\subset\Stilde$ and $|\Stilde| \leq c|S|$, and for every $\Stilde \in\widetilde{\cS}$ there exists $S\in\cS$ with $\Stilde \subset S$ and $|S|\leq \tilde{c}|\Stilde|$.
\end{definition}

The most standard example of equivalent bases 
is that of balls and cubes in $\Rn$.  We calculate the 
explicit constants $c, \tilde{c}$  in this equivalence, 
as they will be used in Section~\ref{sec:SDR}.
\begin{lemma}
\label{lem-BQ} 
The basis $\cB$ is equivalent to the basis $\cQ$. 
\end{lemma}

\begin{proof}
Given a cube $Q$ of sidelength $\ell$, there is a ball $B$ of radius $\sqrt{n}\ell/2$ containing $Q$, so 
$|B|= 2^{-n}n^{\frac{n}{2}}\omega_n|Q|$. In the other direction, given a ball $B$ of radius $r$, there is a cube $Q$ of sidelength $2r$ containing $B$, so $|Q|=2^n\omega_n^{-1}|B|$.
\end{proof}

\begin{definition}
We say a function $f$ has {\em bounded mean oscillation} with respect to a basis $\cS$, denoted $f \in \BMO{\cS}{}(\Omega)$, if $f\in L^1(S)$ for all $S\in\cS$ and
\begin{equation}
\label{eq-bmo}
\|f\|_{\BMO{\cS}{}}:=\sup_{S\in\cS} \cO(f,S)<\infty\, .
\end{equation}
\end{definition}

Since mean oscillation is translation invariant under the 
addition of constants, Eq.~\eqref{eq-bmo} defines a 
seminorm that vanishes on constant functions. If one 
considers $\BMO{\cS}{}(\Omega)$ modulo constants, one obtains a Banach space, as was shown in \cite{dg}. However, as we will see below, for the purpose of rearrangement it is useful to just think of $\BMO{\cS}{}(\Omega)$ as a linear space with a seminorm. The notation $\bmo(\Omega)$ will be reserved for the case $\cS=\cQ$. 

From Eq.~\eqref{O-subset}, it follows that $f\in\BMO{\cS}{}(\R^n)$ if and only if $f\in\BMO{\widetilde{\cS}}{}(\R^n)$ whenever $\cS\approx\widetilde{\cS}$. More precisely, if the comparability constants are $c$ and $\tilde{c}$, then 
\begin{equation}
\label{B-equiv}
c^{-1} \|f\|_{\BMO{\widetilde{\cS}}{}}
\le \|f\|_{\BMO{\cS}{}}\le \tilde{c}
\|f\|_{\BMO{\widetilde{\cS}}{}}\,.
\end{equation}

We point out that a decreasing function $f$ in $\bmo(\R_+)$ 
automatically satisfies a stronger condition, 
called {\em bounded lower oscillation}, denoted $f \in \BLO{}{}(\R_+)$. 
Introduced by Coifman--Rochberg~\cite{cr}, $\BLO{}{}(\Omega)$ 
is the class of $f\in\Loneloc(\Omega)$ such that 
$$
\sup_{Q}\fint_{Q}\!\left(f-\essinf_{Q}f\right)<\infty\,,
$$
where the supremum is taken over all cubes $Q\subset \Omega$. 
It is a strict subset of $\bmo(\Omega)$ and is not closed 
under multiplication by negative scalars. For the decreasing 
rearrangement, any statement about $\fast\in \bmo(\R_+)$ 
can be interpreted in the stronger sense that $\fast\in \BLO{}{}(\R_+)$.
For a reference on BMO functions, see~\cite{ko2}.


\subsection{Rearrangeability in BMO}

In defining the decreasing rearrangement for functions in BMO, 
several issues arise. 
Since mean oscillation 
is invariant under the addition of constants, while 
rearrangement is not, the mapping from $f$ to 
$\fast$ is not a mapping between equivalence classes modulo 
constants, but between individual functions. 
This can be avoided, for BMO functions that are bounded 
below, by considering the map $f \ra (f - \inf f)^\ast$, which 
is well-defined modulo constants.  

Moreover, functions in BMO need not be rearrangeable. 
One example is $-\log|x|$, the prototypical unbounded function 
in $\bmo(\Rn)$. On the other hand, the positive part 
$(-\log|x|)_+$ is rearrangeable, as is any other 
$\bmo$-function of compact support since such functions are integrable.

In this paper, we use Definition~\ref{def-fast}, according to 
which $f^\ast=|f|^*$. Since functions in BMO are 
locally integrable, hence finite almost everywhere, 
a function $f$ is rearrangeable provided that 
$\mu_f(\alpha) < \infty$ 
for some $\alpha \geq 0$. This property is preserved under the
addition of constants.



\section{{\bf Decreasing rearrangement}}
\label{sec:bounds}


We first present a general result from~\cite{BDG2} 
which guarantees the BMO-boundedness of the decreasing 
rearrangement under assumptions on the basis $\cS$. 
We then proceed to derive from it bounds with 
improved constants for several well-known cases, 
culminating in the proof of Theorem~\ref{thm-newbound}.

\subsection{General boundedness criterion}
We start with the definition, in our general setting, of a version of the Calder\'{o}n--Zygmund decomposition, one of most used tools in harmonic analysis. Note the emphasis on the constants.

\begin{definition} 
\label{def-CZ}
Let $f$ be a 
nonnegative measurable function
on $\Omega$ and $c_\ast\geq{1}$. We say that $\cS$ admits a 
{\em $c_\ast$-Calder\'{o}n--Zygmund decomposition} 
for $f$ at a level $\gamma>0$ if there exist a 
pairwise-disjoint sequence $\{S_i\}\subset\cS$ and a corresponding 
sequence $\{\widetilde{S}_i\}\subset\cS$ such that
\begin{itemize}
\item[(i)] 
for all $i$, $\widetilde{S}_i\supset S_i$ and $|\widetilde{S}_i|
\leq c_*|S_i|$;
\item[(ii)] for all $i$, 
$f$ is integrable on $\widetilde S_i$,
with 
$\displaystyle{\fint_{\widetilde{S}_i}\!f \le \gamma
\le \fint_{S_i}\!f}$;
\end{itemize}
and 
\begin{itemize}
\item[(iii)] $f\leq \gamma$ almost 
everywhere on $\Omega\setminus\bigcup \widetilde{S}_i$.
\end{itemize}
\end{definition} 

We now state the boundedness criterion from~\cite{BDG2}. 
The proof given there owes much to the work of Klemes~\cite{kl} in 
dimension one, though it does not yield a sharp
constant since there is no Rising Sun lemma in this generality. 
Note that the constant for the bound equals the 
constant in the Calder\'{o}n--Zygmund decomposition, clearly 
demonstrating the dependence on the geometry of the shapes.

\begin{lemma}[{\cite[Theorem 4.4]{BDG2}}]
\label{lem-bound} 
Let $c_*\geq 1$. Suppose that
for every nonnegative $g\in L^\infty(\Omega)$ and each $t\in (0,|\Omega|)$,
$\cS$ admits a $c_*$-Calder\'on--Zygmund decomposition 
at level $\gamma = \gast_{(0,t)} $.
Then for every rearrangeable function $f\in\BMO{\cS}{}(\Omega)$,
the decreasing rearrangement $f^*$ is locally
integrable, and
$$\|\fast\|_{\bmo}\leq c_*\|f\|_{\BMO{\cS}{}}.$$
\end{lemma}

The boundedness of $g$ implies that 
$g^*\in L^\infty(\R_+)\subset\Loneloc(\R_+)$, 
so $\gamma$ as defined in this lemma is necessarily finite. 
For a general rearrangeable $f \in \BMO{\cS}{}(\Omega)$, we do not know {\em a priori} that $\fast$ is locally integrable and $\fast_{(0,t)}$ is finite.

The following lemma will be used in
the construction  of the families $\{S_i\}$ and 
$\{\Stilde_i\}$ satisfying conditions (i), (ii), and (iii) 
of Definition~\ref{def-CZ}. 

\begin{lemma}\label{localcz}
Let $0\leq g\in \Linfty(\R^n)$, $t>0$,
and $\gamma= g^*_{(0,t)}$. Then
$g_S\leq \gamma$ for any shape $S$ with $|S|\ge t$.
\end{lemma}
\begin{proof}
By the Hardy--Littlewood 
inequality Eq.~\eqref{R-HL},
$$
g_S\leq g^*_{(0,|S|)} \le g^*_{(0,t)}=\gamma\,,
$$ 
where the second inequality holds since $g^*$ is decreasing.
\end{proof}

\subsection{Families of rectangles}
\label{sec-rect_to_cubes}
 We briefly describe
the construction of the Calder\'on--Zygmund 
decompositions required for certain bases of rectangles (Cartesian products 
of intervals) in order for  Lemma~\ref{lem-bound} to hold.

We start with the collection of cubes, $\cQ$,  in $\Omega=\R^n$.
Given $0 \leq g \in L^\infty(\Rn)$ and $t>0$,
we choose a countable collection of cubes 
of measure $t$ that partition $\R^n$ up to a set of measure zero.
By Lemma~\ref{localcz}, the mean of $g$ over
each of these cubes is at most $\gamma$.
On each of these cubes, we follow
the proof of the classical Calder\'{o}n--Zygmund 
decomposition (see~\cite[Theorem 1.3.2]{st2}) to obtain a
$2^n$-Calder\'{o}n--Zygmund decomposition of cubes for 
$g$ at level $\gamma=\fint_{0}^{t}g^*$.
Therefore, Lemma~\ref{lem-bound} implies that every 
rearrangeable $f\in\bmo(\R^n)$ has decreasing 
rearrangement $\fast\in\bmo(\R_+)$ with
$\|\fast\|_{\bmo}\leq 2^{n}\|f\|_{\bmo}$.
This improves on the Bennett--DeVore--Sharpley bound in Eq.~\eqref{eq-bds} by a constant factor,
but falls far short of Theorem~\ref{thm-newbound}.

\smallskip 
For the purpose of comparison, again on $\Omega=\R^n$, consider the basis $\cR$
comprised of rectangles of arbitrary proportions.
Let $0 \leq g \in L^\infty(\Rn)$ and $t>0$. 
Applying Lemma~\ref{localcz} and 
the multidimensional 
analogue of Riesz' Rising Sun lemma of
Korenovskii-Lerner-Stokolos~\cite{kls}, one obtains 
an extreme case of Definition~\ref{def-CZ}, where $c_* = 1$ and 
one has a single countable family 
$\{R_i\} = \{\Rtilde_i\}$ 
that is both pairwise disjoint and on which the averages of $g$ 
can be made equal to $\gamma$.
In this case, Lemma~\ref{lem-bound} implies that 
every rearrangeable $f\in\BMO{\cR}{}(\R^n)$ has decreasing 
rearrangement $\fast\in\bmo(\R_+)$, with
$\|\fast\|_{\bmo}\leq \|f\|_{\BMO{\cR}{}}$, see
Eq.~\eqref{eq-KK}. This inequality
was proved on finite rectangles $R_0$ by Korenovskii~\cite{ko3}.

We point out that when the domain $\Omega$ is a cube $Q_0$ or a rectangle 
$R_0$,  the bound on the rearrangement can be used to obtain a 
John-Nirenberg inequality on $\Omega$ with constants 
depending on $c_*$, in the same way as Korenovskii derived 
the sharp John--Nirenberg inequality on a rectangle. 
In such a case $\Omega$ is itself a shape, and $\cS$ admits a  
$c_*$-Calder\'on--Zygmund decomposition for every 
$g \in L^\infty(\Omega)$ at any level 
$\gamma \leq |g|_\Omega$, which by Lemma~\ref{localcz} 
includes $\gamma = \gast_{(0,t)}$ for $t < |\Omega|$.  

\smallskip 
Simple examples show that an analogue of 
Riesz' Rising Sun lemma cannot exist for cubes~\cite{ko3}. 
Moreover, we cannot deduce from Korenovskii's
sharp inequality any information about bounds on the 
decreasing rearrangement on $\bmo(\R^n)$,
because arbitrary rectangles are not comparable to cubes.

\subsection{Bisection bases}

As a compromise between the rigidity of cubes and the freedom 
of arbitrary rectangles, we
consider particular families of rectangles of bounded 
eccentricity.  If such a family $\cS$ satisfies the assumptions of 
Lemma~\ref{lem-bound} with constant $c_*$,
and moreover $\|f\|_{\bmo_{\cS}}\leq c\,\|f\|_{\bmo}$ 
for some constant $c$, then it follows that
Eq.~\eqref{eq-bds} holds with a constant $C_n\le c_*c$.

To this end, we  introduce a class of bases $\cF$ of rectangles that will be shown to satisfy the assumptions of Lemma~\ref{lem-bound} with constant $c_*=2$, regardless of the dimension.

\begin{definition}
We say that a basis $\cF$ in $\Rn$ is a
\emph{bisection basis}, if for every $S\in\cF$ 
there is a bisection that splits $S$ into two shapes in $\cF$.
\end{definition}

For such bases, we obtain the following dimension-free bound.

\begin{lemma}
\label{lem:bisect}
Let $\cF$ be a bisection basis of rectangles of uniformly 
bounded eccentricity in $\Rn$
that contains, for each $t>0$, a partition 
of $\R^n$ (up to a set of measure zero) into rectangles of measure at least $t$.
If $f \in\BMO{\cF}{}(\R^n)$ is rearrangeable, 
then $f^*\in\BMO{}{}(\R_+)$ with 
$$
\|f^*\|_{\BMO{}{}}\leq 2\|f\|_{\BMO{\cF}{}}.
$$
\end{lemma}

\begin{proof}
We need to show that the basis $\cF$ satisfies the assumptions of Lemma~\ref{lem-bound}. 
Let $0 \leq g \in L^\infty(\Rn)$, $t>0$, and $\gamma=g^*_{(0,t)}$.  
Partition $\R^n$ into rectangles from $\cF$ with measure 
no smaller than $t$. 
By Lemma~\ref{localcz}, the mean of $g$ is at most 
$\gamma$ for each of these rectangles. Bisect each of 
these rectangles, select any subrectangles satisfying 
$g_R\geq\gamma$, and continue to bisect those for which $g_R\leq\gamma$.

The result is a pairwise-disjoint collection of rectangles 
$R_i\in\cF$ satisfying $g_{R_i}\geq\gamma$, with 
their parent rectangles $\widetilde {R_i}\in\cF$ satisfying 
$g_{ \widetilde {R_i} }\leq \gamma$. By construction 
$|\widetilde {R_i}|=2|R_i|$, and the 
the Lebesgue differentiation theorem implies 
that $g\leq \gamma$ a.e. on $\Rn\setminus\bigcup \widetilde {R_i}$.
\end{proof}

As pointed out earlier, a consequence of this result is the following: 
if $\|f\|_{\bmo_{\cF}}\leq c\,\|f\|_{\bmo}$ holds 
with some constant $c>0$ for some bisection basis $\cF$, 
then Eq.~\eqref{eq-bds} holds with $C_n\leq 2c$.

Consider the basis $\cP$ of rectangles with sidelengths
given by a permutation of 
$\{s2^{-j/n}: 0 \leq j \le n-1\}$
for some $s > 0$. For example, when $n = 2$, these correspond 
to rectangles whose sides have lengths $s$ and $s/\sqrt{2}$
(the dimensions of the ISO 216 standard paper size,
which includes the commonly used A4 format).
The basis $\cP$ is a bisection basis (with bisection along the longest side) of rectangles of constant eccentricity, and 
so $\|f^*\|_{\BMO{}{}}\leq 2\|f\|_{\BMO{\cP}{}}$ holds 
for all rearrangeable $f\in \BMO{\cP}{}(\Rn)$ by Lemma~\ref{lem:bisect}. 
Moreover, these rectangles are comparable with cubes, resulting in a relationship $\|f\|_{\bmo_{\cP}}\leq 2^{\frac{n-1}{2}}\|f\|_{\bmo}$. The constant's dependence on dimension is a result of the volume ratio obtained by circumscribing a rectangle in $\cP$ by a cube. This implies an improvement of the constants in Eq.~\eqref{eq-bds} to $C_n\leq 2^{\frac{n+1}{2}}$, but one that does not break past exponential dependence on the dimension.

\smallskip
Consider now the basis $\cW$ of 
false cubes in $\R^n$. As defined in \cite{wik},
a false cube is a rectangle in $\R^n$ 
with sidelengths $2s$ in the coordinate directions
$i=1,\ldots ,m$ and side lengths $s$
in the remaining coordinate directions.
Here, $s>0$, and $1\le m\le n$. This is also a bisection basis: for an arbitrary false cube, bisecting it along the $m$th coordinate results in 
two congruent rectangles in $\cW$. Moreover, rectangles in $\cW$ have bounded eccentricity and so Lemma~\ref{lem:bisect} implies that 
\begin{equation}\label{bisectionconstant}
\|f^*\|_{\BMO{}{}}\leq 2\|f\|_{\BMO{\cW}{}}
\end{equation}
for all rearrangeable $f\in \BMO{\cW}{}(\Rn)$.

In order to obtain an improvement of the constants in  Eq.~\eqref{eq-bds}, 
it remains to investigate the relationship between 
$\|\cdot\|_{\BMO{\cW}{}}$ and $\|\cdot\|_{\bmo}$. 
As a first step, we estimate the oscillation of $f$ over a false cube in terms
of its means over its subcubes.

\begin{lemma} 
\label{lem:partition}
Let $R$ be a false cube, given (up to a set of measure zero) by a disjoint 
union $R=\bigcup_\nu Q(\nu)$  of $2^m$ cubes of equal size.
Then 
$$
\frac{1}{2^m}\sum_\nu |f_{Q(\nu)}-f_R| 
\ \le \ \cO(f, R)
\ \le \ \|f\|_\bmo + \frac{1}{2^m}\sum_\nu |f_{Q(\nu)}-f_R| 
$$
for all $f\in\bmo(\R^n)$.
\end{lemma}

\begin{proof} For the lower bound, we write
$$
\cO(f, R) 
\ = \ \sum_\nu \frac{|Q(\nu)|}{|R|} \fint_{Q(\nu)}\!|f-f_R|
\ \ge \ \frac{1}{2^m} \sum_\nu |f_{Q(\nu)}-f_R|\,.
$$
For the upper bound, we use the triangle inequality
to obtain
\begin{align*}
\cO(f, R) 
&\le  \frac{1}{2^m}\sum_\nu \fint_{Q(\nu)}\!
\bigl( |f-f_{Q(\nu)}| + |f_{Q(\nu)}-f_R|\bigr)\\
&\le \sup_\nu
\cO(f, Q(\nu)) + \frac{1}{2^m}\sum_\nu |f_{Q(\nu)}-f_R|\,,
\end{align*}
and bound the first term by $\|f\|_\bmo$.  
\end{proof}

The next lemma gives a dimension-free bound on the difference 
between the means of a function $f\in\bmo(\R^n)$ 
on a pair of neighbouring cubes. In~\cite[Lemma 3]{wik}, Wik obtains the analogous estimate for the medians, with the constant 6.

\begin{lemma}
\label{lem:neighbors}
Let $Q_1,Q_2$ be two cubes of equal size in
$\R^n$ that share an $(n-1)$-dimensional face.
Then $|f_{Q_1}-f_{Q_2}|\le 4 \|f\|_\bmo$.
\end{lemma}

\begin{proof}
Let $Q_0$ be the cube of the same size as $Q_1, Q_2$ that is bisected by
the common interface. We will prove that
$$
|f_{Q_j}-f_{Q_0}|\le 2 \|f\|_\bmo\,,\qquad j=1,2\,,
$$
and then apply the triangle inequality.

Let $A_1=Q_1\setminus Q_0$, $A_0=Q_0\setminus Q_1$, and
$B_1=Q_1\cap Q_0$. By construction, $|A_1|=|A_0|=|B_1|=\frac12|Q_1|$.
By linearity of the mean,
$$ 
f_{Q_j}= \frac12(f_{A_j}+f_{B_1})\,,\quad j=0,1\,.
$$
On the other hand, in the same way as in the proof of Lemma~\ref{lem:partition},
$$
\frac12(|f_{A_j}-f_{Q_j}|+|f_{B_1}-f_{Q_j}|)\le\cO(f,Q_j)\le \|f\|_\bmo\,,
\quad j=0,1\,.
$$
It follows that
$$
|f_{A_j}-f_{Q_j}|=|f_{B_1}-f_{Q_j}|\le \|f\|_\bmo\,,
\quad j=0,1\,,
$$
and we conclude
that
$$
|f_{Q_1}-f_{Q_0}|\le |f_{Q_1}-f_{B_1}| +|f_{B_1}-f_{Q_0}|\le 2\|f\|_\bmo\,.
$$ 
By the same argument, $|f_{Q_2}-f_{Q_0}|\le 2\|f\|_\bmo$.
\end{proof}

We make use of the following concentration inequality.
\begin{lemma}
\label{lem:concentration}
Let 
$\{X_1,\ldots,X_m\}$ be $m$ independent 
random variables on a probability space, 
and let $g=g(X_1,\ldots,X_m)$ be a random variable
of finite expectation.
If there are constants $a_1,\dots, a_m$ such that
\begin{equation}\label{conditiononf}
|g(x_1,\ldots,x_i,\ldots,x_m)-g(x_1,\ldots,x_i'\ldots,x_m)|\leq a_i
\end{equation}
for all $x_i$, $x_i'$ and all $1\leq i\leq m$, then
\begin{equation}
\label{eq:concentration}
\E\left|g-\E g\right| \le \frac12\|a\|_2\,,
\end{equation}
where $\|a\|_2=\left(\sum a_i^2\right)^{\frac12}$ 
is the Euclidean length of the vector $(a_1,\ldots, a_m)\in\R^m$.
\end{lemma}

\begin{proof} 
We estimate the variance
of $g(X_1,\dots, X_m)$ using a martingale argument. 
Let $\mu_i$ be the probability distribution of $X_i$.
The difference $g-\E g$ can be written as a telescoping sum,
$$
g-\E g =\sum_{i=1}^m (g_i-g_{i-1})\,,
$$
where 
$g_i(x_1,\ldots, x_i):=\E g(x_1,\ldots, x_i, X_{i+1},\ldots X_m)$.
By construction, $g_0=\E g$ and $g_m=g$.

Consider the differences
$Y_i:= g_i(X_1,\ldots, X_i)-g_{i-1}(X_1,\dots, X_{i-1})$.
By independence, 
$$
\int\! g_i(x_1,\ldots , x_i)\,{d}\mu_i(x_i)=g_{i-1}(x_1,\ldots, x_{i-1})\,,
$$
which implies $\E Y_i=0$. Moreover, 
for each pair $i<j$, we have by a similar calculation
$$
\int\! \bigl(g_i(x_1,\ldots, x_i)-g_{i-1}(x_1,\ldots, x_{i-1})\bigr)
\bigl(g_j(x_1,\ldots, x_j)-g_{j-1}(x_1,\ldots, x_{j-1})\bigr)\,
d\mu_j(x_j)=0\,,
$$
which implies that $ {\rm Cov}\, (Y_i,Y_j)= \E(Y_iY_j)=0$
for $1\le i<j\le m$.

We next estimate the variance of $Y_i$.
By Eq.~\eqref{conditiononf}, we have for
any $x_1,\ldots, x_i$,
\begin{align*}
\int\! \bigl(g_i(x_1,\dots, x_i)-g_{i-1}(x_1,\ldots, x_{i-1})\bigr)^2\, d\mu_i(x_i)
&= \inf_{c\in\R} 
\int 
\bigl(g_i(x_1,\ldots, x_i)-c\bigr)^2\, d\mu_i(x_i)\\
&\le \frac14 a_i^2\,,
\end{align*}
The first inequality holds by definition of $g_{i-1}$
as  a mean, and the second inequality follows
by choosing $c$ as the arithmetic mean of the
extremal values of $g(x_1,\ldots, x_{i-1},\cdot)$.
Integrating over the remaining variables, we obtain
${\rm Var}\, Y_i\leq  \frac14 a_i^2$.

We conclude, using that the random variables $Y_i$ are uncorrelated, that
$$
\E (g-\E g)^2
={\rm Var}\, g(X_1,\dots, X_m)=\sum_{i=1}^m {\rm Var}\, Y_i\le \frac14 
\|a\|_2^2\,.
$$
By Schwarz' inequality, this implies Eq.~\eqref{eq:concentration}.
\end{proof}

The previous three lemmas combine to give a simple proof of Eq.~\eqref{eq:falsecompare}.

\begin{lemma}
\label{Wik}
Let $f\in\BMO{}{}(\R^n)$. Then 
$\|f\|_{\BMO{\cW}{}} \leq (1+2\sqrt{n-1}) \|f\|_{\BMO{}{}}$.
\end{lemma}

\begin{proof}
Let $f\in \BMO{}{}(\R^n)$ and 
$R$ be an arbitrary false cube
that has sidelength $2s$ in the first $m$ directions, 
$1\leq m \leq n$, and sidelength $s$ in the 
remaining $n-m$ dimensions. If $m=n$, then $R$ is a cube and 
there is nothing to show.

If $1\le m<n$, we write $R$ as a 
union of $2^m$ subcubes $Q(\nu)$ indexed by binary strings 
$\nu\in \{0,1\}^m$. By Lemma~\ref{lem:partition},
$$\fint_{R}|f-f_R| 
\le \|f\|_\bmo + \frac{1}{2^m}\!\sum_{\nu\in \{0,1\}^m}
|f_{Q(\nu)}-f_R|\,.
$$

Set $g(\nu):=f_{Q(\nu)}$ for $\nu\in \{0,1\}^m$. We consider $g$ as a random variable on the
space of binary strings $\{0,1\}^m$, equipped with
the uniform measure.
and verify the hypotheses of Lemma~\ref{lem:concentration}.
Note that $f_R=\E g$ by the linearity of expectation.
The components $(X_1,\dots, X_m)$ of a random 
binary string are Bernoulli random variables with bias $p=\frac12$.
By Lemma~\ref{lem:neighbors},
$$
|f_Q-f_{Q'}|\le  4\|f\|_\bmo
$$
for any pair of adjacent cubes.
Therefore, Eq.~\eqref{conditiononf} holds with $a_i=4\|f\|_\bmo$
for $i=1,\ldots m$, and we obtain from Lemma~\ref{lem:concentration} that
$\E|g-\E g|\le 2\sqrt{m}\|f\|_{\bmo}$;
that is,
$$
\frac{1}{2^m}\!\sum_{\nu\in \{0,1\}^m}
|f_{Q(\nu)}-f_R|\le 2\sqrt{m}\|f\|_\bmo\,.
$$
Since $m\le n-1$, this concludes the proof.
\end{proof}

We finally have all the necessary components to prove Theorem~\ref{thm-newbound}.

\begin{proof}[Proof of Theorem~\ref{thm-newbound}]
This follows from a combination of Eq.~\eqref{bisectionconstant} and Lemma~\ref{Wik}. 
\end{proof}


\section{{\bf Symmetric decreasing rearrangement}}
\label{sec:SDR}


This section is dedicated to the proof of 
Theorem~\ref{thm-SDRequivalent}. 
We consider BMO with respect to a basis $\cA$ that is adapted
to the radial symmetry of $Sf$. This basis consists
of all balls centred at the origin, along with
annular sectors obtained by intersecting an annulus with a 
suitable spherical cone. 

For $x\in\R^n\setminus\{0\}$, $0<\rho\leq|x|$, 
and $\alpha\in \left(0,\frac{\pi}{2}\right]$, write
$$
A(x,\rho,\alpha):=\{y\in\R^n:
|x|-\rho<|y|<|x|+\rho,\, y\cdot x >|x|\, |y|\,\cos\alpha\}\,.
$$
The set $A(x,\rho,\alpha)$ is the sector of 
width $2\rho$ and aperture $\alpha$ centred at $x$. Then 
$$
\cA=\bigl\{B(0,r): r>0\bigr\}\cup 
\left\{A(x,\rho,\alpha):\rho=|x|\sin\alpha,\alpha\in\left(0,\frac{\pi}{2}\right]\right\}\,
$$
is a basis of shapes in $\R^n$. 

This basis captures those sets on which the symmetric decreasing 
rearrangement $Sf$ can be compared directly with $f^*$. More 
generally, $\bmo(\R_+)$ is isometric to
the subspace of radial functions in $\BMO{\cA}{}(\R^n)$.

\begin{lemma}
\label{lem-AI}
Let $f$ be a radial function on $\Rn$, defined by
$f(x)=g(\omega_n|x|^n)$. Then, 
$f\in \BMO{\cA}{}(\R^n)$ if and only if $g\in \bmo(\R_+)$, with
$$
\|f\|_{\BMO{\cA}{}} = \|g\|_{\bmo}\,.
$$
\end{lemma}

\begin{proof}
Integration over any shape $A\in\cA$ can be represented by an integral
in spherical coordinates over $J\times S$, where
$J=(a,b)\subset\R_+$ is an interval and $S$ is 
either the entire sphere or a spherical cap.
Let $\nu$ be the measure on $\R_+$
with density $n \omega_n r^{n-1}$. 
Since $f$ is radial, two changes of variables give
\[
f_A =\fint_{A}\!f = \fint_J\! g(\omega_n r^n) d\nu(r) = \fint_I\! g
= g_I\,,
\]
where $I=(\omega_n a^n, \omega_n b^n)$ (see~\cite[Corollary 2.51]{fol}).
In the same way, $\cO(f,A)=\cO(g, I)$,
from where it follows that $\|f\|_{\BMO{\cA}{}}\leq \|g\|_{\bmo}$.

Conversely, given an open interval $I\subset(0,\infty)$, set
$J=\{s>0: \omega_ns^n\in I\}$. We distinguish between
two cases. If the left endpoint of $J$ is at $0$, 
we take $A$ to be the ball of measure $|I|$ centred at the origin.
Otherwise,  we choose 
$A=A(R e_1,\rho, \alpha)$, where $R$ is the centre of 
$J$, $\rho=\frac12 |J| < R$,
$\alpha=\arcsin(\rho/R)$, and $e_1$ is the standard unit vector $e_1=(1,0,\ldots,0)$.
Reversing the calculation above, we
see that $\cO(g,I)=\cO(f,A)$, and hence
$\|g\|_{\bmo}\le \|f\|_{\BMO{\cA}{}}$.
\end{proof}

Consider two rearrangeable functions $f_1,f_2$. The 
function $Sf_1-Sf_2$ 
is radial, and satisfies $(Sf_1-Sf_2)(x)=(f_1^*-f_2^*)(\omega_n|x|^n)$ 
by the definition of the symmetric decreasing rearrangement. 
Lemma~\ref{lem-AI} implies that 
\begin{equation}\label{SDR-AI}
\|Sf_1-Sf_2\|_{\BMO{\cA}{}} =\|f_1^*-f_2^*\|_{\bmo} \,.
\end{equation}
We next compare $\cA$ with the basis of balls,
$\cB$, and hence with the standard basis $\cQ$.

\begin{lemma}
\label{lem-AB} 
The basis $\cA$ is equivalent to the basis $\cB$. 
\end{lemma}

\begin{figure}[t]
\includegraphics[width=0.8 \linewidth]{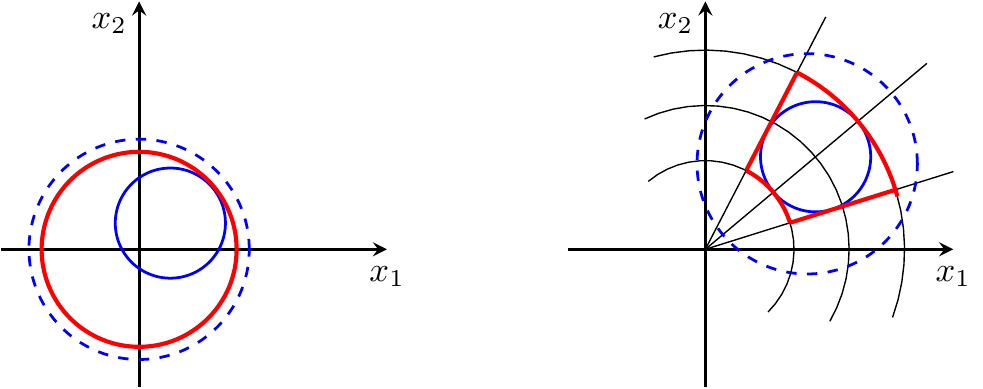}
\caption{\small Equivalence of the bases $\cA$ and $\cB$.}
\end{figure}

\begin{proof} 
\smallskip We first show that for each ball
$B$ there exist $A\in \cA$ and $\widetilde B\in \cB$ such that 
\begin{equation}
\label{shape-AB}
B \subset A\subset \tilde B\quad \text{and}
\quad |A|<|\widetilde B|= 2^n |B|\,,
\end{equation} 
see Figure~1.

Let a ball $B=B(x,r)$ be given. If $|x|<r$, then $B$ contains the origin.
In this case, we choose $A$ to be the ball of radius $|x|+r$ centred at the origin and $\widetilde B$ to be the ball of radius $2r$ centred at the origin. It follows that $|A| < |\widetilde B|=2^n|B|$,
where the inequality is strict because $A$ is a proper subset of $\widetilde B$. Since $A$ exhausts $\widetilde B$ as $|x|$ approaches $r$ from below, the constant $2^n$ is sharp.

If $|x|\ge r$, then the ball does not contain the origin.
In that case, we take $A=A(x,r,\alpha)$
where  $\alpha = \arcsin\bigl(\tfrac{r}{|x|}\bigr)$,
and $\widetilde B= B(\widetilde x\,, 2r)$,
where $\widetilde x=(\cos\alpha)\, x$.
Clearly, $A\in \cA$. We begin by showing that $B\subset A$. To that end, fix $y\in B$. The triangle inequality implies that 
\begin{equation}\label{annulus}
|x|-r< |x|-|x-y| \leq|y|\leq |x|+|x-y| <|x|+r\,.
\end{equation}
Writing $x=|x|\,\xi$ and $y=|y|\,\eta$, where
$\eta$ and $\xi$ are unit vectors, it follows that
$$
r^2 > |y-x|^2= |x|^2+ |y|^2-2|x|\, |y|\,\xi\cdot \eta 
\ge |x|^2(1-(\xi\cdot\eta)^2)\,.
$$
Hence,
\begin{equation}\label{dotproduct}
(\xi\cdot \eta)^2>1-\frac{r^2}{|x|^2}=1-\sin^2\alpha= \cos^2\alpha\,.
\end{equation}
Since $B$ does not include the origin, the angle between any vector terminating in $B$ and the vector terminating at the centre of $B$ is at most $\pi/2$. Thus, $x\cdot y>0$ and so $\xi\cdot \eta>0$. It follows from this and Eq.~\eqref{dotproduct} that $\xi\cdot \eta>\cos\alpha$. Together with Eq.~\eqref{annulus}, this implies that $y\in A$, demonstrating that $B\subset A$. 

To see that $A\subset\widetilde B$,
we bound the distance to any point $y=|y|\,\eta\in A$
from $\widetilde x=(\cos\alpha)|x|\,\xi$. 
Since both $x\in B\subset A$ and 
$y \in A=A(x,r,\alpha)$, we have $\xi\cdot \eta>\cos\alpha$.
Using that $|\widetilde x|^2=|x|^2-r^2$, we calculate
$$
|\widetilde x-y|^2
< |x|^2-r^2 +|y|^2 -2|x||y|\cos^2\alpha 
= (|x|-|y|)^2 - r^2 + 2|x||y|\sin^2\alpha\,.
$$
Since $|x|-r<|y|<|x|+r$ and $|x|\,\sin\alpha=r$,
it follows that
$$
(|x|-|y|)^2 - r^2 + 2|x||y|\sin^2\alpha< 2|x|(|x|+r)\sin^2\alpha 
< 4r^2\,.
$$
Therefore $A\subset B(\widetilde x, 2r)=\widetilde B$.  
This
proves Eq.~\eqref{shape-AB}.

\smallskip In the other direction,
given $A\in \cA$, we will construct $B,\widetilde B\in \cB$ such that 
\begin{equation}
\label{shape-BA}
B \subset A\subset \widetilde B\quad \text{and}
\quad |\widetilde B| =2^n|B|\leq 2^n|A|\,.
\end{equation} 
If $A$ is a centred ball, we take $B=A$ 
and $\widetilde B$ to be the centred ball of twice the radius. If $A= A(x,\rho,\alpha)$ with $\rho=|x|\,\sin\alpha$, 
we take $B=B(x,\rho)$ and $\widetilde B=B(\widetilde{x}, 2\rho)$, 
where $\widetilde{x}=(\cos\alpha)\, x$. 
The same calculations as above
show that Eq.~\eqref{shape-BA} holds for this choice 
of $B$ and $\widetilde B$.
\end{proof}

We can now establish the 
bi-Lipschitz equivalence between rearrangements.

\begin{proof}[Proof of Theorem~\ref{thm-SDRequivalent}]
Let $f_1, f_2$ be rearrangeable
functions. From Lemmas~\ref{lem-BQ} and \ref{lem-AB}, along with 
Eq.~\eqref{B-equiv}, it follows that 
$$
2^{-n}\omega_n\|Sf_1-Sf_2\|_{\BMO{\cB}{}}\leq \|Sf_1-Sf_2\|_\bmo\le 2^{-n} n^{\frac{n}{2}} \omega_n\|Sf_1-Sf_2\|_{\BMO{\cB}{}}
$$
and 
$$
2^{-n}\|Sf_1-Sf_2\|_{\BMO{\cA}{}}\leq \|Sf_1-Sf_2\|_{\BMO{\cB}{}}\leq 2^n \|Sf_1-Sf_2\|_{\BMO{\cA}{}}\,.
$$
Combining these with Eq.~\eqref{SDR-AI} yields 
Eq.~\eqref{eq-SDRequivalent}.
\end{proof}

We also have the following local version 
of Theorem~\ref{thm-SDRequivalent}.

\begin{lemma}
\label{SDRlocalequivalent}
Let $R > 0$ and $Q\subset B(0,R)$
be a cube of diameter $d$, centred at a 
point $x$ with $|x|\le R-d/2$.
There is an interval $I\subset (0,\omega_n R^n)$ of length
$|I|\le n \omega_nR^{n-1}d$, 
such that if $f_1,f_2$ are rearrangeable, then 
$$
\cO(Sf_1-Sf_2,Q)\leq n^{\frac{n}{2}}\omega_n\,\cO(f_1^*-f_2^*,I)\,.
$$
\end{lemma}

\begin{proof}
Let $Q$ be as in the statement of the lemma.
Lemmas~\ref{lem-BQ} and~\ref{lem-AB}, 
along with Eq.~\eqref{O-subset}, imply that
$$
\cO(Sf_1-Sf_2,Q)\leq n^{\frac{n}{2}}\omega_n\cO(Sf_1-Sf_2,A)\,,
$$
where $A\in\cA$ is constructed, as in the proof of Lemma~\ref{lem-AB}, from the ball $B(x,d/2)$ and satisfies $Q \subset B \subset A\subset B(0,R)$ and
$|A|\le n^{\frac{n}{2}}\omega_n|Q| = \omega_n d^n$. 
As in the proof of Lemma~\ref{lem-AI}, $A$ is represented
in polar coordinates by
$J\times S$, where $J\subset (0,R)$ is an interval of length at most $d$,
and $S$ lies in the unit sphere.
Moreover,
$$
\cO(Sf_1-Sf_2,A)=\cO(f_1^*-f_2^*,I)\,,
$$
where the interval $I$ is the image of $J$ under the map
$r\mapsto \omega_n r^n$. 
The length of this interval is bounded by $|I|\le n\omega_n R^{n-1}|J|$.
\end{proof}

\section*{{Acknowledgements}}

The authors became aware of Wik's paper thanks to comments by Michael Cwikel following the talk of Carlos P\'{e}rez in the Corona Seminar.


\end{document}